\documentclass[12pt]{article}

\usepackage{tikz}
\usepackage{subfigure}
\usepackage[english]{babel}

\usepackage[center]{caption2}
\usepackage{amsfonts,amssymb,amsmath,latexsym,amsthm}
\usepackage{multirow}

\def\ex{\mbox{ex}}

\newtheorem{thm}{Theorem}[section]

\newtheorem{lem}[thm]{Lemma}

\newtheorem{quest}[thm]{Question}
\newtheorem{claim}{Claim}

\newtheorem{cons}{Construction}

\begin{document}

\title{Codegree threshold for tiling $k$-graphs with two edges sharing exactly $\ell$ vertices\thanks{The work was supported by NNSF of China (no. 11671376) and NNSF of Anhui Province (no. 1708085MA18).}}
\author{Lei Yu, \quad Xinmin Hou\\
\small  Key Laboratory of Wu Wen-Tsun Mathematics\\
%\small Chinese Academy of Sciences\\
\small School of Mathematical Sciences\\
\small University of Science and Technology of China\\
\small Hefei, Anhui 230026, China.\\
%\small $^{c}$Department of Mathematics
}
\date{}
\maketitle

\begin{abstract}
%Given two $k$-graphs $F$ and $H$, an $F$-factor in $H$ is a set of vertex-disjoint copies of $F$ which cover all the vertices in $H$.
Given integer $k$ and a $k$-graph $F$,  let $t_{k-1}(n,F)$ be the minimum integer $t$ such that every $k$-graph $H$ on $n$ vertices with codegree at least $t$ contains an $F$-factor. For integers $k\geq3$ and $0\leq\ell\leq k-1$, let $\mathcal{Y}_{k,\ell}$ be a $k$-graph with two edges that shares exactly $\ell$ vertices. Han and Zhao (JCTA, 2015) asked the following question: For all $k\ge 3$, $0\le \ell\le k-1$ and sufficiently large $n$ divisible by $2k-\ell$, determine the exact value of $t_{k-1}(n,\mathcal{Y}_{k,\ell})$. In this paper, we show that $t_{k-1}(n,\mathcal{Y}_{k,\ell})=\frac{n}{2k-\ell}$ for $k\geq3$ and $1\leq\ell\leq k-2$, combining with two previously known results of  R\"{o}dl, Ruci\'{n}ski and Szemer\'{e}di {(JCTA, 2009)} and Gao, Han and Zhao (arXiv, 2016), the question of Han and Zhao is solved completely.
%for sufficiently large $n\in (2k-\ell)\mathbb{N}$, every $k$-uniform hypergraph on $n$ vertices with minimum codegree at least $\frac{n}{2k-\ell}$ contains a $\mathcal{Y}_{k,\ell}$-factor. This codegree condition is best possible and answers a question of .

\end{abstract}

\section{Introduction}
Given $k\geq2$, a $k$-uniform hypergraph ($k$-graph for short) consists of a vertex set $V$ and an edge set $E\subseteq\binom{V}{k}$, where $\binom{V}{k}$ is the set of all $k$-element subsets of $V$. Let $H$ be a $k$-graph  and let $S\subset V(H)$ with $|S|=d$ ($1\leq d\leq k-1$). The degree of $S$, denoted by $\deg_H(S)$, is  the number of edges containing $S$ (the subscript $H$ will be omitted if it is clear from the context). The minimum
$d$-degree $\delta_d(H)$ of $H$ is the minimum of $\deg_H(S)$ over all $d$-element vertex sets $S$ in $H$. We refer to $\delta_1(H)$ and $\delta_{k-1}(H)$ as the minimum degree and codegree of $H$, respectively.

Given two hypergraphs $H$ and $F$, an {\it $F$-tiling} in $H$ is a collection of vertex-disjoint copies of $F$ in $H$. An $F$-tiling is called {\it perfect} if it covers all the vertices of $H$. Perfect $F$-tilings are also referred to as {\it $F$-factors}. Given a $k$-graph $F$ and integer $n$ divisible by $|F|$, let $t_{k-1}(n,F)$ be the minimum integer $t$ such that every $k$-graph $H$ on $n$ vertices with $\delta_{k-1}(H)\geq t$ contains an $F$-factor, we also {call} $t_{k-1}(n,F)$ the codegree threshold of $F$.

Given $k\geq3$ and $0\leq\ell\leq k-1$, let $\mathcal{Y}_{k,\ell}$ be a $k$-graph with two edges that shares exactly $\ell$ vertices. In this paper, we mainly concern the codegree threshold of $\mathcal{Y}_{k,\ell}$. Some special cases of $t_{k-1}(n, \mathcal{Y}_{k,\ell})$ have been obtained in literatures. R\"{o}dl, Ruci\'{n}ski and Szemer\'{e}di~\cite{RRS-JCTA09} determined the exact value of $t_{k-1}(n, \mathcal{Y}_{k,\ell})$ when $\ell=0$, more precisely, they proved that for all $k\geq3$ and sufficiently large $n$ divisible {by $2k$},
\begin{equation}\label{EQN:e1}
t_{k-1}(n, \mathcal{Y}_{k,0})=\frac n2-k+c,
\end{equation}
where {$c\in\{2, 3\}$};
%{is determined by explicit divisibility condition on $n$ and $k$};
%the tiling of which is coincide with a perfect matching, this significant result is due to .
for $k=3$ and $\ell=2$, K\"{u}hn and Osthus~\cite{KO-JCTB06} showed that $t_2(n, \mathcal{Y}_{3,2})=n/4 +o(n)$, the exact value of  $t_2(n, \mathcal{Y}_{3,2})$ was given by Czygrinow, DeBiasio and Nagle~\cite{CDN-JGT14};
as a generalization, Gao, Han and Zhao~\cite{GHZ-arXiv16} determined the exact value of $t_{k-1}(n,\mathcal{Y}_{k,\ell})$ for all $k\ge 3$ and $\ell=k-1$, i.e. they proved that for any $k\ge 3$ and sufficiently large $n$ divisible by $k + 1$,
\begin{equation}\label{EQN: e2}
 t_{k-1}(n, \mathcal{Y}_{k,k-1})=\frac{n}{k+1}+c,
\end{equation}
where $c\in\{0,1\}$;
for general $k\ge 3$ and {$1\le \ell\le k-1$}, by a result on tiling $k$-partite $k$-graphs given by Mycroft~\cite{Myc-JCTA15}, we have $t_{k-1}(n, \mathcal{Y}_{k,\ell})=\frac{n}{2k-\ell}+o(n)$,   in~\cite{HZ-JCTA15}, Han and Zhao constructed an extremal graph for $\mathcal{Y}_{k,\ell}$, which yields that $t_{k-1}(n, \mathcal{Y}_{k,\ell})>\frac{n}{2k-\ell}-1$, and in the same paper, the authors asked the following question.
\begin{quest}[\cite{HZ-JCTA15}]\label{QUES: q1}
For all $k\ge 3$, $0\le \ell\le k-1$ and sufficiently large $n$ divisible by $2k-\ell$, determine the exact value of $t_{k-1}(n,\mathcal{Y}_{k,\ell})$.
\end{quest}
In this paper, we give the exact value of  $t_{k-1}(n,\mathcal{Y}_{k,\ell})$ for $k\geq3$, $1\leq\ell\leq k-2$ and sufficiently large $n$, combining with (\ref{EQN:e1}) and (\ref{EQN: e2}), Question~\ref{QUES: q1} is answered completely.

\begin{thm}\label{THM: Main}
For all $k\geq3$, $1\leq\ell\leq k-2$ and sufficient large $n$ divisible by $2k-\ell$,
$$t_{k-1}(n,\mathcal{Y}_{k,\ell})=\frac{n}{2k-\ell}.$$
\end{thm}

\begin{cons}[Extremal graph, \cite{HZ-JCTA15}]\label{CONS: A}
 Let $H_0$ be a $k$-graph on $n\in (2k-\ell)\mathbb{N}$ vertices such that {$V(H_0)=A\dot\cup B$} with $|A|=\frac{n}{2k-\ell}-1$, and $E(H_0)$ consists of all $k$-subsets of $A\cup B$ intersecting $A$ and some $k$-subsets of $B$ such that $H_0[B]$ contains no copy of $\mathcal{Y}_{k,\ell}$.
\end{cons}
Clearly, $\delta_{k-1}(H_0)\geq \frac{n}{2k-\ell}-1$. Since every copy of $\mathcal{Y}_{k,\ell}$ contains at least one vertex in $A$, there is no $\mathcal{Y}_{k,\ell}$-factor in $H_0$.

The proof of Theorem~\ref{THM: Main} follows the clue given by Han and Zhao in~\cite{HZ-JCTA15}, that is we use the standard "absorbing method", which has been widely used in study of tiling problems (see for example~\cite{HZ-JCTA15,CDN-JGT14,RRS-JCTA09,XLM-18,GHZ-arXiv16}). As pointed by Han and Zhao in~\cite{HZ-JCTA15}, to determine the exact value of $t_{k-1}(n,\mathcal{Y}_{k,\ell})$,  it suffices to prove an absorbing lemma and the extremal case.
Fortunately, the absorbing lemma given in~\cite{GHZ-arXiv16} does work here and so our main contribution in this paper is to deal with the extremal case.

We give more definitions and notation which will be used in the paper. Let $H$ be a $k$-graph, write $e(H)$ or $|H|$ for the size of $E(H)$. For a set $A\subseteq V(H)$, let $H[A]$ be the subgraph induced by $A$ and denote $e_H(A)=|H[A]|$, and $\overline{e}_H(A)=\binom{|V(A)|}{k}-e_H(A)$. The subscript will be omitted if the underlying hypergraph is clear from the context. For two vertex sets $S, R\subseteq V(H)$ with $|S|<k$, let $N_H(S,R)=\{T :  T\subseteq R  \text{ such that $S\cap T=\emptyset$ and } S\cup T\in E(H)\}$ and $\deg_H(S,R)=|N_H(S, R)|$.
%for the number of $(k-|S|)$-sets $T\subseteq R$ such that $S\cup T\in E(H)$, %such a $T$ is called a neighbor of $S$,
Define $\overline{\deg}_H(S,R)=\binom{|R\setminus S|}{k-|S|}-\deg_H(S,R)$, the number of non-edges in $S\cup R$ that contain $S$. By the definitions here,  $\deg_H(S)=\deg_H(S, V(H))$ and $\overline{\deg}_H(S)=\overline{\deg}_H(S, V(H))$. If $S=\{v\}$, write $\deg_H(v,R)$ and $\overline{\deg}_H(v,R)$ for $\deg_H(\{v\},R)$ and $\overline{\deg}_H(\{v\},R)$, respectively.  We say $H$ is {\it $\xi$-extremal} if there exists a set $B\subseteq V(H)$ of size $(1-\frac 1{2k-\ell})n$ such that $e_H(B)\leq \xi \binom{|B|}{k}$. In the paper, for constants $\alpha, \beta$, $\alpha\ll \beta$ means $\alpha$ is small enough compared to $\beta$.

The rest of the paper is arranged as follows. In Section 2, we give lemmas and the proof of Theorem~\ref{THM: Main}. The extremal case lemma will be proved in  Section $3$.

\section{Proof of Theorem~\ref{THM: Main}}
%Using the typical approach of obtaining exact results, our proof of Theorem~\ref{THM: Main} consists of an extremal case and a nonextremal case.
To cope with the non-extremal case, we need an absorbing lemma and an almost tiling lemma for $\mathcal{Y}_{k,\ell}$.
In~\cite{GHZ-arXiv16}, Gao, Han and Zhao gave an absorbing lemma (Lemma 3.1) for general complete $k$-partite $k$-graphs, as a special case, we have the absorbing lemma for $\mathcal{Y}_{k,\ell}$.

\begin{lem}\label{LEM: absorbing}(Absorbing Lemma)
Let $k\geq 3$, $1\leq\ell\leq k-2$, suppose $0<\alpha\ll\gamma\ll \frac{1}{2k-\ell}$ and $n$ is sufficiently large. If $H$ is an $n$-vertex $k$-graph such that $\delta_{k-1}(H)\geq \frac{n}{2k-\ell}$, then there exists a vertex set $W\subseteq V(H)$ with $|W|\leq \gamma n$ and {$|W|\in (2k-\ell)\mathbb{N}$} such that for any vertex set $U\subseteq V(H)\backslash W$ with $|U|\leq\alpha n$ and $|U|\in (2k-\ell)\mathbb{N}$, both $H[W]$ and $H[U\cup W]$ contain $\mathcal{Y}_{k,\ell}$-factors.
\end{lem}

The almost tiling lemma used here also is a special case of the $\mathcal{Y}_{k,\ell}$-tiling lemma (Lemma 2.8) given by Han and Zhao in~\cite{HZ-JCTA15} and a special case of the almost tiling lemma for general $k$-partite $k$-graphs (Lemma $3.2$) given by Gao, Han and Zhao in~\cite{GHZ-arXiv16}.

\begin{lem}\label{LEM: tiling}(Almost tiling Lemma)
Let $k\geq 3$, $1\leq\ell\leq k-2$, for any $\alpha,\gamma,\xi>0$ such that $\gamma\ll \xi$, there exists an integer $n_0$ such that the followimg holds. If $H$ is a $k$-graph on $n>n_0$ vertices with $\delta_{k-1}(H)\geq (\frac{1}{2k-\ell}-\gamma)n$, then $H$ has a $\mathcal{Y}_{k,\ell}$-tiling that covers all but at most $\alpha n$ vertices unless $H$ is $\xi$-extremal.
\end{lem}

%Also, Lemma~\ref{LEM: tiling} is a special case of Lemma $2.8$ in ~\cite{HZ-JCTA15} and Lemma $3.2$ in~\cite{GHZ-arXiv16}. Following the proof of Lemma $2.8$ in ~\cite{HZ-JCTA15}, we can even obtain a $\mathcal{Y}_{k,\ell}$-tiling which leaves only constant vertices uncovered under the same conditions.

Our contribution in the proof of Theorem~\ref{THM: Main} is to give the lemma of  extremal case for $\mathcal{Y}_{k,\ell}$. The proof will be given in the next section.

\begin{lem}[Extremal case]\label{LEM: extremal}
Given $k\geq 3$, $1\leq\ell\leq k-2$, $0<\xi \ll {\frac{1}{2k-\ell}}$ and let $n\in (2k-\ell)\mathbb{N}$ be sufficiently large. Suppose $H$ is a $k$-graph on $n$ vertices with $\delta_{k-1}(H)\geq \frac{n}{2k-\ell}$. If $H$ is $\xi$-extremal, then $H$ contains a $\mathcal{Y}_{k,\ell}$-factor.
\end{lem}

Now we give the proof of Theorem~\ref{THM: Main}.
\begin{proof}[Proof of Theorem~\ref{THM: Main}:]
Let $k\geq 3$, $1\leq\ell\leq k-2$ and let $n\in (2k-\ell)\mathbb{N}$ be sufficiently large. Choose $\alpha$ and $\gamma$ small enough such that $0<\alpha\ll\gamma\ll \xi \ll \frac{1}{2k-\ell}$.  Suppose $H$ is an $n$-vertex $k$-graph satisfying $\delta_{k-1}(H)\geq \frac{n}{2k-\ell}$. If $H$ is $\xi$-extremal, then $H$ contains a $\mathcal{Y}_{k,\ell}$-factor by Lemma~\ref{LEM: extremal}. Otherwise, by Lemma~\ref{LEM: absorbing}, we find an absorbing set $W$ in $V(H)$ of size at most $\gamma n$ which has the absorbing property. Let $H':= H-W$ and $n'=|V(H')|\geq (1-\gamma)n$. If $H'$ is $\frac{\xi}{2}$-extremal, then there exists a $B'\subseteq V(H')$ of order $(1-\frac 1{2k-\ell})n'$ such that $e_{H'}(B')\leq \frac{\xi}{2}\binom{|B'|}{k}$. Thus by adding to $B'$ at most $n-n'\leq \gamma n$ vertices, we get a set $B$ of size precisely $(1-\frac{1}{2k-\ell})n$ in $V(H)$ with
$$e_{H}(B)\leq e_{H'}(B')+\gamma n\binom{n-1}{k-1}\leq \frac{\xi}{2}\binom{|B'|}{k}+k\gamma\binom{n}{k}\leq \xi\binom{|B|}{k},$$
a contradiction to the assumption that $H$ is {not} $\xi$-extremal. So we assume that $H'$ is not $\frac{\xi}{2}$-extremal. Since
$$\delta_{k-1}(H')\geq \frac{n}{2k-\ell}-\gamma n\geq (\frac{1}{2k-\ell}-\gamma)n',$$
applying Lemma~\ref{LEM: tiling} on $H'$ with $\frac{\xi}{2}$, we obtain a $\mathcal{Y}_{k,\ell}$-tiling $\mathcal{Y}$ that covers all but a set $U$ of at most $\alpha n$ vertices. {Since both $n$ and $|W|$ are divisible by $2k-\ell$, $|U|\in (2k-\ell)N$.} By the absorbing property of $W$, $H[W\cup U]$ contains a $\mathcal{Y}_{k,\ell}$-factor and together with the $\mathcal{Y}_{k,\ell}$-tiling $\mathcal{Y}$ we obtain a $\mathcal{Y}_{k,\ell}$-factor of $H$.
\end{proof}

%\section{Proof of the almost perfect tiling lemma}
%Given a $k$-graph $H$ with a set $S$ of at most $k-1$ vertices , the link graph of $S$ is the $(k-|S|)$-graph with vertex set $V(H)\backslash S$ and edge set ${e\backslash S:e\in E(H), S\subseteq e}$.

%\begin{proof}[Proof of Lemma~\ref{LEM: tiling}]
%Given $\alpha,\gamma,\xi$ such that $\gamma\ll 1$ and let $n\in \mathbb{N}$ be sufficiently large. Let $H$ be a $k$-graph on $n$ vertices that satisfies $\delta_{k-1}(H)\geq (\frac{1}{2k-\ell}-\gamma)n$, which is not $\xi$-extremal. We prove that $H$ contains a $\mathcal{Y}_{k,\ell}$-tiling covering all but at most $\alpha n$ vertices. To that end, let $\mathcal{Y}=\{\mathcal{Y}_1,...,\mathcal{Y}_m\}$ be a largest $\mathcal{Y}_{k,\ell}$-tiling in $H$(with respect to $m$) and write $V_i=V(\mathcal{Y}_i)$ fot $i\in [m]$. Let $V'=\bigcup_{i\in [m]}V_i$ and $U=V(H)\backslash V'$. On the contrary, assume that $|U|>\alpha n$.
%Let $C$ be the set of vertices $v\in V'$ such that $deg(v,U)\geq (2k-\ell)^2\binom{|U|}{k-2}$.

\section{Proof of Lemma~\ref{LEM: extremal}}
We need more definitions and  notation in the proof.
Given two disjoint sets $X,Y$ and two integers $i,j\geq 0$, a set $S\subset X\cup Y$ is called of type $X^iY^j$ if $|S\cap X|=i$ and $|S\cap Y|=j$.
%Define $K(X^iY^j)$ be the $(i+j)$-graph on vertex set $X\cup Y$ and edge set consisting of all sets of type $X^iY^j$.
If $X$ and $Y$ are two disjoint vertex subsets of a $k$-graph $H$ and $i+j=k$, denote by $H(X^iY^j)$ the subgraph induced by all edges of type $X^iY^j$ in $H$ and let $e_H(X^iY^j)=|H(X^iY^j)|$ and $\overline{e}_H(X^iY^{j})={|X|\choose i}{|Y|\choose j}-e_H(X^iY^j)$ (the subscript may be omitted if it is clear from the context).
%We use $\overline{e}_H(X^iY^{k-i})$ to denote the number of non-edges among $X^iY^{k-i}$-sets.
Given a set $L\subseteq X\cup Y$ with $|L\cap X|=l_1\leq i$ and $|L\cap Y|=l_2\leq k-i$, define $\deg(L,X^iY^{k-i})$ be the degree of $L$ in $H(X^iY^{k-i})$ and $\overline{\deg}(L,X^iY^{k-i})=\binom{|X|-l_1}{i-l_1}\binom{|Y|-l_2}{k-i-l_2}-\deg(L,X^iY^{k-i})$.
%, i.e. the degree of $L$ in $K(X^iY^{k-i})-E(H(X^iY^{k-i}))$.

Given two $k$-graphs $F$ and $H$, we call $H$ {\it $F$-free} if $H$ does not contain $F$ as a subgraph. The well-known Tur\'{a}n number $\ex_k(n,F)$ is the maximum number of edges in an $F$-free $k$-graph on $n$ vertices. The following result was given by Frankl and F\"{u}redi~\cite{FF-JCTA85}.

{\begin{lem}[\cite{FF-JCTA85}]\label{THM: Turan}
For $k\geq 2$, $0\leq\ell\leq k-1$, there exists a constant $d_k$ depending only on $k$ such that $\ex_k(n,\mathcal{Y}_{k,\ell})\leq d_k n^{max\{\ell,k-\ell-1\}}$.
\end{lem}}
%{\color{blue}\begin{lem}[\cite{FF-JCTA85}]\label{THM: Turan}
%For $k\geq2$, $0\leq\ell\leq k-1$, there exists a constant $d_k$ depending only on $k$ such that $\ex_k(n,\mathcal{Y}_{k,\ell})\leq d_k n^{max\{\ell,k-\ell-1\}}$.
%\end{lem}}

The following lemma also is a special version of a result (Lemma 6.1 in~\cite{GHZ-arXiv16} ) given by Gao, Han and Zhao.
\begin{lem}[\cite{GHZ-arXiv16}]\label{LEM: partition}
Given $k\geq 3$, $1\leq\ell\leq k-2$. Let $0<\rho \ll \frac{1}{2k-\ell}$ and let $n$ be sufficiently large. Suppose $H$ is a $k$-graph on $n\in (2k-\ell)\mathbb{N}$ vertices with a partition of $V(H)=X\cup Y$ such that $|Y|=(2k-\ell-1)|X|$. Furthermore, assume that

 (a) for every vertex $v\in X$, $\overline{\deg}(v,Y)\leq \rho\binom{|Y|}{k-1}$,

 (b) for every vertex $u\in Y$, $\overline{\deg}(u,XY^{k-1})\leq \rho\binom{|Y|}{k-1}$.\\
Then $H$ contains a $\mathcal{Y}_{k,\ell}$-factor.
\end{lem}

%The following proof uses the standard techniques to cope with the extremal case    the idea of the proof of Theorem $3.4$ in~\cite{GHZ-arXiv16}, we can get the proof of Lemma~\ref{LEM: extremal}.

\begin{proof}[Proof of Lemma~\ref{LEM: extremal}]
%Assume that $k\geq 3$, $1\leq\ell\leq k-2$, $0<\xi \ll \frac{n}{2k-\ell}$ and let $n\in (2k-\ell)\mathbb{N}$ be sufficiently large. Let $H$ be a $k$-graph on $V$  of $n$ vertices such that $\delta_{k-1}(H)\geq \frac{n}{2k-\ell}$. Furthermore, assume that
Since $H$ is $\xi$-extremal, there is a set $B\subseteq V(H)$ such that $|B|=(1-\frac{1}{2k-\ell})n$ and $e(B)\leq \xi \binom{|B|}{k}$. Let $A=V(H)\setminus B$. Then $|A|=\frac{n}{2k-\ell}$. Let $\epsilon_1=\xi^{\frac{1}{3}}$, $\epsilon_2=2\epsilon_1^2=2\xi^{\frac{2}{3}}$. Define
\begin{eqnarray*}
A'&:=&\left\{v\in V(H)\ :\  {\deg}(v,B)\geq(1-\epsilon_1)\binom{|B|}{k-1}\right\},\\
B'&:=&\left\{v\in V(H)\ :\ {\deg}(v,B)\leq\epsilon_1\binom{|B|}{k-1}\right\},
\end{eqnarray*}
and
$$V_0:=V(H)\setminus(A'\cup B').$$

\begin{claim}\label{2}
$\{|A\setminus A'|,|B\setminus B'|,|A'\setminus A|,|B'\setminus B|\}\leq\epsilon_2|B|$ and $|V_0|\leq 2\epsilon_2|B|$.
\end{claim}

\begin{proof}[Proof of Claim~\ref{2}] First assume that $|B\backslash B'|>\epsilon_2|B|$. By the definition of $B'$, we have
$$e(B)>\frac{1}{k}\epsilon_2|B|\cdot\epsilon_1\binom{|B|}{k-1}>2\xi\binom{|B|}{k},$$
a contradiction to $e(B)\leq \xi \binom{|B|}{k}$.

Second, assume that $|A\setminus A'|>\epsilon_2|B|$. By the definition of $A'$, for any vertex $v\notin A'$, we have  $\overline{\deg}(v,B)>\epsilon_1\binom{|B|}{k-1}$. So
$$\overline{e}(AB^{k-1})>\epsilon_2|B|\cdot\epsilon_1\binom{|B|}{k-1}=2\xi |B|\binom{|B|}{k-1}.$$
Together with $e(B)\leq \xi \binom{|B|}{k}$, we have
\begin{eqnarray*}
\sum_{S\in \binom{B}{k-1}}\overline{\deg}(S)&=&k\cdot\overline{e}(B)+\overline{e}(AB^{k-1})\\
&>&k(1-\xi)\binom{|B|}{k}+2\xi|B|\binom{|B|}{k-1}\\
&=&\left[(1-\xi)(|B|-k+1)+2\xi |B|\right]\binom{|B|}{k-1}\\
&>&|B|\binom{|B|}{k-1},
\end{eqnarray*}
where the last inequality holds because  $n$ is sufficiently large. By the pigeonhole principle, there exists a set $S\in \binom{B}{k-1}$, such that $\overline{\deg}(S)>|B|=(1-\frac{1}{2k-\ell})n$, a contradiction to $\delta_{k-1}(H)\geq \frac{n}{2k-\ell}$.

Consequently,
$$|A'\setminus A|=|A'\cap B|\leq |B\setminus B'|\leq\epsilon_2|B|,$$
$$|B'\setminus B|=|A\cap B'|\leq |A\setminus A'|\leq\epsilon_2|B|,$$
$$|V_0|\le |A\setminus A'|+|B\setminus B'|\leq\epsilon_2|B|+\epsilon_2|B|\leq2\epsilon_2|B|.$$
\end{proof}

 By $|B\setminus B'|\leq\epsilon_2|B|$ and $|B'\setminus B|\leq\epsilon_2|B|$,  for any vertex $v\in V_0$, we have
$$\deg(v,B')\geq \deg(v,B)-|B\setminus B'|\binom{|B|}{k-2}{\geq \frac{\epsilon_1}{2}\binom{|B'|}{k-1},}$$
%{\color{blue}$\Leftarrow \epsilon_1\binom{|B|}{k-1}-\epsilon_2|B|\binom{|B|}{k-2}\geq \frac{\epsilon_1}{2}\binom{(1+\epsilon_2)|B|}{k-1}\\
%\Leftarrow \epsilon_1|B|...(|B|-k+2)-(k-1)\epsilon_2|B||B|...(|B|-k+3)\geq \frac{\epsilon_1}{2}(1+\epsilon_2)|B|...[(1+\epsilon_2)|B|-k+2]\\
%\Leftarrow \epsilon_1-(k-1)\epsilon_2\frac{|B|}{|B|-k+2}\geq \frac{\epsilon_1}{2}(1+\epsilon_2)...(1+\frac{|B|}{|B|-k+2}\epsilon_2)\\
%\Leftarrow \epsilon_1-2(k-1)\epsilon_2\geq \frac{\epsilon_1}{2}(1+2\epsilon_2)^{k-1}\\
%\Leftarrow \epsilon_1-4(k-1)\epsilon_1^2\geq \frac{\epsilon_1}{2}(1+4\epsilon_1^2)^{k-1}\\
%\Leftarrow 2\geq 8(k-1)\epsilon_1+(1+4\epsilon_1^2)^{k-1}$\\}

for any vertex $v\in A'$,
$$\overline{\deg}(v,B')\leq \overline{\deg}(v,B)+|B'\setminus B|\binom{|B'|}{k-2}\leq {2\epsilon_1\binom{|B'|}{k-1},}$$
%{\color{blue}$\Leftarrow \epsilon_1\binom{|B|}{k-1}+\epsilon_2|B|\binom{|B'|}{k-2}\leq 2\epsilon_1\binom{|B'|}{k-1}\\
%\Leftarrow \epsilon_1\binom{|B|}{k-1}+\epsilon_2|B|\binom{(1+\epsilon_2)|B|}{k-2}\leq 2\epsilon_1\binom{(1-\epsilon_2)|B|}{k-1}\\
%\Leftarrow \epsilon_1|B|...(|B|-k+2)+(k-1)\epsilon_2|B|(1+\epsilon_2)|B|...[(1+\epsilon_2)|B|-k+3)]\leq 2\epsilon_1(1-\epsilon_2)|B|...[(1-\epsilon_2)|B|-k+2)]\\
%\Leftarrow \epsilon_1+(k-1)\epsilon_2\frac{|B|}{|B|-k+2}(1+\epsilon_2)...(1+\frac{|B|}{|B|-k+3}\epsilon_2)\leq 2\epsilon_1(1-\epsilon_2)...(1-\frac{|B|}{|B|-k+2}\epsilon_2) \\
%\Leftarrow \epsilon_1+2(k-1)\epsilon_2(1+2\epsilon_2)^{k-2}\leq 2\epsilon_1(1-2\epsilon_2)^{k-1} \\
%\Leftarrow 1+4(k-1)\epsilon_1(1+2\epsilon_2)^{k-2}\leq 2(1-2\epsilon_2)^{k-1}$}\\

and for any vertex $v\in B'$,
$$\deg(v,B')\leq \deg(v,B)+|B'\setminus B|\binom{|B'|}{k-2}\leq {2\epsilon_1\binom{|B'|}{k-1}}.$$
Moreover, for any $(k-1)$-set $S\subseteq B'$, since $\deg(S,A')+\deg(S,B')+\deg(S,V_0)\geq \delta_{k-1}(H)$ and $\overline{\deg}(S,A')=|A'|-\deg(S,A')$, we have
$$\overline{\deg}(S,A'){\le}|A'|-\delta_{k-1}(H)+\deg(S,B')+\deg(S,V_0)\leq \deg(S,B')+3\epsilon_2|B|,$$
where the last inequality holds since $\deg(S,V_0)\leq |V_0|\leq 2\epsilon_2|B|$, $|A'|\leq \frac{n}{2k-\ell}+\epsilon_2|B|$ and $\delta_{k-1}(H)\geq \frac{n}{2k-\ell}$. Furthermore, for any $v\in B'$, we have
$$\sum_{S: v\in S\in \binom{B'}{k-1}}\deg(S,B')=(k-1)\deg(v,B')\leq2(k-1)\epsilon_1\binom{|B'|}{k-1}.$$
Putting this together gives that for any $v\in B'$,
\begin{eqnarray*}
\overline{\deg}(v,A'(B')^{k-1})&=&\sum_{S: v\in S\in \binom{B'}{k-1}}\overline{\deg}(S,A')\\
&\leq& \sum_{S: v\in S\in \binom{B'}{k-1}}\deg(S,B')+3\epsilon_2|B|\binom{|B'|-1}{k-2}\\
&\leq& 2k\epsilon_1\binom{|B'|}{k-1}.
\end{eqnarray*}
%{\color{blue}$\Leftarrow 2(k-1)\epsilon_1\binom{|B'|}{k-1}+3\epsilon_2|B|\binom{|B'|-1}{k-2}\leq 2k\epsilon_1\binom{|B'|}{k-1}\\
%\Leftarrow 2(k-1)\epsilon_1|B'|...(|B'|-k+2)+3(k-1)\epsilon_2|B|(|B'|-1)...(|B'|-k+2)\leq 2k\epsilon_1|B'|...(|B'|-k+2)\\
%\Leftarrow 2(k-1)\epsilon_1|B'|+3(k-1)\epsilon_2|B|\leq 2k\epsilon_1|B'|\\
%\Leftarrow 3(k-1)\epsilon_2|B|\leq 2\epsilon_1|B'|\\
%\Leftarrow 3(k-1)\epsilon_2|B|\leq 2\epsilon_1(1-\epsilon_2)|B|\\
%\Leftarrow 1\geq 3(k-1)\epsilon_1+\epsilon_2\\$}

%where the sums are on $S$ such that $v\in S\in \binom{B'}{k-1}$.
%{\color{red} (Please check these inequalities, seemingly that  $\epsilon_2$ should be smaller enough than $\epsilon_1$, XM0712) }
So, if $|B'|=(2k-\ell-1)|A'|$ and $|V_0|=\emptyset$ then applying Lemma~\ref{LEM: partition} we obtain a $\mathcal{Y}_{k,\ell}$-factor of $H$.
%such that each copy of $\mathcal{Y}_{k,\ell}$ has $1$ vertex in $A'$ and $2k-\ell-1$ vertices in $B'$.
%In the following,  we will build four vertex-disjoint $\mathcal{Y}_{k,\ell}$-tilings $Y_1,Y_2,Y_3,Y_4$, whose union is a $\mathcal{Y}_{k,\ell}$-factor of $H$. The ideal case is when $|B'|=(2k-\ell-1)|A'|$ and $|V_0|=\emptyset$, in this case we apply Lemma~\ref{LEM: partition} to obtain a $\mathcal{Y}_{k,\ell}$-factor of $H$ such that each copy of $\mathcal{Y}_{k,\ell}$ has $1$ vertex in $A'$ and $2k-\ell-1$ vertices in $B'$.
%So the purpose of the $\mathcal{Y}_{k,\ell}$-tilings $Y_1,Y_2,Y_3$ is to cover the vertices of $V_0$ and adjust the sizes of $A'$ and $B'$ so that we can apply Lemma~\ref{LEM: partition} and obtain $Y_4$ after $Y_1,Y_2,Y_3$ are removed.

Now we assume $|B'|\not=(2k-\ell-1)|A'|$ or $|V_0|\not=\emptyset$.
Let $q:=|B'|-|B|=\frac{n}{2k-\ell}-|A'|-|V_0|.$ Then $-\epsilon_2|B|\leq q\leq \epsilon_2|B|.$

%{\color{red}The first step, we find $q$ vertex-disjoint copies of $\mathcal{Y}_{k,\ell}$ in $H[B']$. If $q\leq0$, set $Y_1:=\emptyset$. If $q>0$, we claim that we can greedily construct $q$ vertex-disjoint copies of $\mathcal{Y}_{k,\ell}$ in $H[B']$. In fact, suppose that we have found $i$ copies of $\mathcal{Y}_{k,\ell}$ for some $0\le i<q$ and let $U$ be the set of the vertices of $B'$ covered by these $i$ copies of $\mathcal{Y}_{k,\ell}$. Then $|U|\leq (2k-\ell)(q-1)$. Since $\deg(v,B')\leq  2\epsilon_1\binom{|B'|}{k-1}$ for any vertex $v\in B'$ and $\delta_{k-1}(H[B'])\geq q$,
%$$e(B'\setminus U)\geq \frac{q}{k}\binom{|B'|}{k-1}-(2k-\ell)(q-1)2\epsilon_1\binom{|B'|}{k-1}>\ex_k(|B'|-|U|,\mathcal{Y}_{k,\ell}),$$
%where the last inequality holds because $n$ is sufficiently large and $\epsilon_1$ is small enough. By Lemma~\ref{THM: Turan}, we can find a copy of $\mathcal{Y}_{k,\ell}$ avoiding $U$. The claim holds.  Set $Y_1$ be the $q$ vertex-disjoint copies of $\mathcal{Y}_{k,\ell}$ in $H[B']$.}

{The first step, we find $q$ vertex-disjoint copies of $\mathcal{Y}_{k,\ell}$ in $H[B']$ when $q>0$. We claim that we can greedily construct $q$ vertex-disjoint copies of $\mathcal{Y}_{k,\ell}$ in $H[B']$. In fact, suppose that we have found $i$ copies of $\mathcal{Y}_{k,\ell}$ for some $0\le i<q$ and let $U$ be the set of the vertices of $B'$ covered by these $i$ copies of $\mathcal{Y}_{k,\ell}$. Then $|U|\leq (2k-\ell)(q-1)$. Since $\deg(v,B')\leq  2\epsilon_1\binom{|B'|}{k-1}$ for any vertex $v\in B'$ and $\delta_{k-1}(H[B'])\geq q$,
$$e(B'\setminus U)\geq \frac{q}{k}\binom{|B'|}{k-1}-(2k-\ell)(q-1)2\epsilon_1\binom{|B'|}{k-1}>\ex_k(|B'|-|U|,\mathcal{Y}_{k,\ell}),$$
where the last inequality holds because $n$ is sufficiently large and $\epsilon_1$ is small enough. By Lemma~\ref{THM: Turan}, we can find a copy of $\mathcal{Y}_{k,\ell}$ avoiding $U$. The claim holds.  Set $Y_1$ be the $q$ vertex-disjoint copies of $\mathcal{Y}_{k,\ell}$ in $H[B']$. If $q\leq0$, set $Y_1:=\emptyset$.}

The next step,  we choose a $\mathcal{Y}_{k,\ell}$-tiling $Y_2$ such that each copy of $\mathcal{Y}_{k,\ell}$ contains one vertex in $V_0$ and $2k-\ell-1$ vertices in $B'$. Let $V_0=\{w_1,\ldots,w_{|V_0|}\}$. We claim that, for each $w_i$, we can find a copy of $\mathcal{Y}_{k-1,\ell-1}$ in the $(k-1)$-graph $N(w_i,B')$ such that these $|V_0|$ copies of $\mathcal{Y}_{k-1,\ell-1}$ are vertex disjoint and are also vertex disjoint from $V(Y_1)$. This is possible because the total number of vertices in $B'$ that we need to avoid is at most
$$|V(Y_1)|+(2k-\ell-1)|V_0|\leq \epsilon_2|B|(2k-\ell)+(2k-\ell-1)2\epsilon_2|B|\leq 3(2k-\ell)\epsilon_2|B|,$$
and so we have
{$$|N(w_i,B')|-3(2k-\ell)\epsilon_2|B|\binom{|B'|}{k-2}\geq \frac{\epsilon_1}{3}\binom{|B'|}{k-1}.$$}
%{\color{blue}$$deg(w_i,B')-3(2k-\ell)\epsilon_2|B|\binom{|B'|}{k-2}\geq \frac{\epsilon_1}{3}\binom{|B'|}{k-1}.$$}
%{\color{blue}$\Leftarrow \frac{\epsilon_1}{2}\binom{|B'|}{k-1}-3(2k-\ell)\epsilon_2|B|\binom{|B'|}{k-2}\geq \frac{\epsilon_1}{3}\binom{|B'|}{k-1}\\
%\Leftarrow \frac{\epsilon_1}{2}|B'|...(|B'|-k+2)-3(k-1)(2k-\ell)\epsilon_2|B||B'|...(|B'|-k+3)\geq \frac{\epsilon_1}{3}|B'|...(|B'|-k+2)\\
%\Leftarrow \frac{\epsilon_1}{2}(|B'|-k+2)-3(k-1)(2k-\ell)\epsilon_2|B|\geq \frac{\epsilon_1}{3}(|B'|-k+2)\\
%\Leftarrow \frac{\epsilon_1}{6}(|B'|-k+2)\geq 3(k-1)(2k-\ell)\epsilon_2|B|\\
%\Leftarrow \frac{1}{6}((1-\epsilon_2)|B|-k+2)\geq 6(k-1)(2k-\ell)\epsilon_1|B|\\
%\Leftarrow [\frac{1}{6}-\frac{1}{6}\epsilon_2-6(k-1)(2k-\ell)\epsilon_1]|B|\geq \frac{1}{6}(k-2)$\\}
By Lemma~\ref{THM: Turan}, $N(w_i,B')$ contains a desired $\mathcal{Y}_{k-1,\ell-1}$.  Note that the copy of $\mathcal{Y}_{k-1,\ell-1}$ union $\{w_i\}$ spans a copy of $\mathcal{Y}_{k,\ell}$ in $H$. Therefore, the $|V_0|$ copies of $\mathcal{Y}_{k,\ell}$ form the desired $Y_2$.

Now reset $B_1$ to the set of vertices in  $B'$ not covered by $Y_1\cup Y_2$, $A_1=A'$ and $V_1=A_1\cup B_1$. The third step we choose a $\mathcal{Y}_{k,\ell}$-tiling $Y_3$ to adjust the sizes of $A_1$ and $B_1$ such that $|B_1\setminus V(Y_3)|=(2k-\ell-1)|A_1\setminus V(Y_3)|$. Let $p=\frac{1}{2k-\ell}|V_1|-|A_1|$. Note that $|Y_1|=q$ if $q>0$ and 0 otherwise, $|Y_2|=|V_0|$ and $|V_1|=n-(2k-\ell)(|Y_1|+|V_0|)$. We have
$$p=\frac{n}{2k-\ell}-|Y_1|-|V_0|-|A_1|=q-|Y_1|.$$
If $q>0$ then $p=0$. Thus $|B_1|=|V_1|-|A_1|=(2k-\ell-1)|A_1|$. Therefore, we choose $Y_3=\emptyset$ in this case. Now assume $q\leq0$. Then $Y_1=\emptyset$ and so $p=q\geq -\epsilon_2|B|$. We claim that we can pick $-p$ vertex disjoint copies of $\mathcal{Y}_{k,\ell}$ such that each of them contains two vertices in $A_1$ and $2k-\ell-2$ vertices in $B_1$.
%and denote by $Y_3$ as the set of copies of $\mathcal{Y}_{k,\ell}$.
In fact, for any pair $\{u_i,v_i\}\subseteq  A_1\ (i\leq -p)$,
we show that we can find a copy of $\mathcal{Y}_{k-1,\ell}$ in the $(k-1)$-graph $N(u_i,B_1)\cap N(v_i,B_1)$ such that these $-p$ copies of $\mathcal{Y}_{k-1,\ell}$ are vertex disjoint.
since
%{\color{red}$$|B_1|\geq |B'|-|V(Y_1\cup Y_2)|\geq |B'|-3\epsilon_2|B|(2k-\ell)>(1-\epsilon_1)|B'|.$$
%$\rightarrow$
$$|B_1|\geq |B'|-|V(Y_2)|\geq |B'|-2\epsilon_2|B|(2k-\ell)>(1-\epsilon_1)|B'|.$$
we have, for any $v\in A_1$,
$$\overline{\deg}(v,B_1)\leq \overline{\deg}(v,B')\leq 2\epsilon_1\binom{|B'|}{k-1}<3\epsilon_1\binom{|B_1|}{k-1}.$$
Thus
%{\color{red}$$N(u_i,B_1)\cap N(v_i,B_1)\geq (1-6\epsilon_1)\binom{|B_1|}{k-1}.$$}
{$$|N(u_i,B_1)\cap N(v_i,B_1)|\geq (1-6\epsilon_1)\binom{|B_1|}{k-1}.$$}
Since the total number of vertices in $B_1$ that we need to avoid is at most $(2k-\ell-2)(-p)\leq 2k\epsilon_2|B|$,  we have
%{\color{red}$$N(u_i,B_1)\cap N(v_i,B_1)- 2k\epsilon_2|B|\binom{|B_1|}{k-2}\geq{\frac{1}{2}\binom{|B_1|}{k-1}}.$$}
{$$|N(u_i,B_1)\cap N(v_i,B_1)|- 2k\epsilon_2|B|\binom{|B_1|}{k-2}\geq{\frac{1}{2}\binom{|B_1|}{k-1}}.$$}
%{\color{blue}$\Leftarrow (1-6\epsilon_1)\binom{|B_1|}{k-1}-2k\epsilon_2|B|\binom{|B_1|}{k-2}\geq \frac{1}{2}\binom{|B_1|}{k-1}\\
%\Leftarrow (1-6\epsilon_1)|B_1|...(|B_1|-k+2)-2k(k-1)\epsilon_2|B||B_1|...(|B_1|-k+3)\geq \frac{1}{2}|B_1|...(|B_1|-k+2)\\
%\Leftarrow (1-6\epsilon_1)(|B_1|-k+2)-2k(k-1)\epsilon_2|B|\geq \frac{1}{2}(|B_1|-k+2)\\
%\Leftarrow (\frac{1}{2}-6\epsilon_1)(|B_1|-k+2)\geq2k(k-1)\epsilon_2|B|\\
%\Leftarrow (\frac{1}{2}-6\epsilon_1)(1-\epsilon_1)(1-\epsilon_2)|B|-(\frac{1}{2}-6\epsilon_1)(k-2)\geq2k(k-1)\epsilon_2|B|\\
%\Leftarrow [(\frac{1}{2}-6\epsilon_1)(1-\epsilon_1)(1-\epsilon_2)-2k(k-1)\epsilon_2]|B|\geq(\frac{1}{2}-6\epsilon_1)(k-2)$\\}
By Lemma~\ref{THM: Turan}, we can greedily find $-p$ desired copies of $\mathcal{Y}_{k-1,\ell}$. Note that each copy of $\mathcal{Y}_{k-1,\ell}$ union its corresponding pair $\{u_i,v_i\}$ spans a copy of $\mathcal{Y}_{k,\ell}$. These $-p$ copies of $\mathcal{Y}_{k,\ell}$ forms the claimed  $\mathcal{Y}_{k,\ell}$-tiling, say $Y_3$.
Let $A_2=A_1\setminus V(Y_3)$ and $B_2=B_1\setminus V(Y_3)$. Then
\begin{eqnarray*}
|B_2|&=& |B_1|+(2k-\ell-2)p\\
    &=&  |V_1|-|A_1|+(2k-\ell)p-2p\\
    &=& (2k-\ell)(p+|A_1|)+(2k-\ell)p-|A_2|\\
    &=& (2k-\ell)(|A_1|+2p)-|A_2|\\
    &=&(2k-\ell-1)|A_2|.
\end{eqnarray*}

The last step, we show that $H[A_2\cup B_2]$ contains a $\mathcal{Y}_{k,\ell}$-factor $Y_4$.  Since $|Y_3|\le -p\leq \epsilon_2|B|$, we have
$$|B_2|\geq |B'|-|V(Y_1)\cup V(Y_2)\cup V(Y_3)|\geq |B'|-4\epsilon_2|B|(2k-\ell)>{(1-\epsilon_1)|B'|}.$$
%{\color{blue}$\Leftarrow \epsilon_1|B'|\geq 4(2k-\ell)\epsilon_2|B|\\
%\Leftarrow \epsilon_1(1-\epsilon_2)|B|\geq 4(2k-\ell)\epsilon_2|B|\\
%\Leftarrow 1\geq8(2k-\ell)\epsilon_1+\epsilon_2$\\}
Hence, for every $v\in A_2$,
$$\overline{\deg}(v,B_2)\leq \overline{\deg}(v,B')\leq 2\epsilon_1\binom{|B'|}{k-1}\leq2\epsilon_1\binom{\frac{1}{1-\epsilon_1}|B_2|}{k-1}<3k\epsilon_1\binom{|B_2|}{k-1},$$
and for every $v\in B_2$,
$$\overline{\deg}(v,A_2B_2^{k-1})\leq \overline{\deg}(v,A'(B')^{k-1})\leq 2k\epsilon_1\binom{|B'|}{k-1}\leq 3k\epsilon_1\binom{|B_2|}{k-1}.$$
Apply Lemma~\ref{LEM: partition} to $H[A_2\cup B_2]$ with $X=A_2$, $Y=B_2$ and $\rho=3k\epsilon_1$, we get a $\mathcal{Y}_{k,\ell}$-factor $Y_4$ of $H[A_2\cup B_2]$.

Finally, the union  $Y_1\cup Y_2\cup Y_3\cup Y_4$ forms a $\mathcal{Y}_{k,\ell}$-factor of $H$. This concludes the proof of Lemma~\ref{LEM: extremal}.

\end{proof}

{}

%{\color{blue}\begin{thebibliography}{99}

%\bibitem{CDN-JGT14}
%A. Czygrinow, L. DeBiasio, B. Nagle, Tiling $3$-uniform hypergraphs with $K_4^3-2e$. Journal of Graph Theory,
%75(2):124-136(2014).

%\bibitem{FF-JCTA85}
%P. Frankl, Z. F\"{u}redi, Forbidding just one intersection. J. Combin. Theory Ser. A, 39(2): 160-176 (1985).

%\bibitem{GHZ-arXiv16}
%W. Gao, J. Han, Y. Zhao, Codegree conditions for tiling complete $k$-partite $k$-graphs and loose cycles. preprint, arXiv:1612.07247, (2016).

%\bibitem{HZ-JCTA15}
%J. Han, Y. Zhao, Minimum codegree threshold for Hamilton $\ell$-cycles in $k$-uniform hypergraphs. J. Combin. Theory Ser. A, 132, 194-223, (2015).

%\bibitem{XLM-18}
%X. Hou, B. Liu, Y. Ma, Codegree conditions for tilling balanced complete 3-partite 3-graphs and generalized 4-cycles, submitted.

%\bibitem{KO-JCTB06}
%D. K\"uhn, D. Osthus, Loose Hamilton cycles in 3-uniform hypergraphs of high minimum degree. J. Combin. Theory Ser. B, 96(6)767-821(2006).

%\bibitem{Myc-JCTA15}
%R. Mycroft, Packing k-partite k-uniform hypergraphs. J. Combin. Theory Ser. A, 138, 60-132 (2016).

%\bibitem{RRS-JCTA09}
%V. R\"{o}dl, A. Ruci\'{n}ski, and E. Szemer\'{e}di, Perfect matchings in large uniform hypergraphs with large minimum collective degree. J Combin. Theory Ser. A, 116(3) 613-636(2009).

%\end{thebibliography}}

\end{document}